\documentclass{amsart}
\usepackage{bm}
\usepackage{amsthm}

\newtheorem{theorem}{Theorem}[section]
\newtheorem*{theorem*}{Theorem}

\newtheorem{lemma}[theorem]{Lemma}
\newtheorem{proposition}[theorem]{Proposition}
\newtheorem*{proposition*}{Proposition}

\newcommand{\dist}{\text{d}}

\DeclareMathOperator{\en}{end}
\DeclareMathOperator{\B}{B}
\DeclareMathOperator{\Ends}{Ends}

\theoremstyle{remark}

\theoremstyle{definition}
\newtheorem{definition}[theorem]{Definition}
\newtheorem*{definition*}{Definition}
\newtheorem{example}[theorem]{Example}

\begin{document}

\title{Sequential ends of metric spaces}

\author{M.~DeLyser}
\address{Saint Francis University, Loretto, PA 15940}
\email{mrd100@@francis.edu}

\author{B.~LaBuz}
\address{Saint Francis University, Loretto, PA 15940}
\email{blabuz@@francis.edu}

\author{M.~Tobash}
\address{Saint Francis University, Loretto, PA 15940}
\email{mst100@@francis.edu}

\begin{abstract}
We develop an analog to the ends of a metric space for the category of coarse metric spaces and show that it is equivalent to a previously defined coarse invariant.
\end{abstract}

\maketitle
\tableofcontents

\section{Introduction}
The papers \cite{borno}, \cite{coarse}, and \cite{direct} develop a coarse invariant of metric spaces (Definition \ref{SigmaDef} in this paper). Given a metric space $X$ we define, using coarse functions $f:\mathbb N\to X$, a coarse analog to the ends of $X$ that we call the sequential ends (Definition \ref{SequentialEnds}). We show the set of sequential ends is equivalent to the above invariant (Theorem \ref{MainTheorem}).

We work in the coarse category (for metric spaces) defined by Roe \cite{Roe}. A function $f:X\to Y$ between metric spaces is \textit{bornologous} if for each $N>0$ there is an $M>0$ such that for every $x,y\in X$, if $\dist(x,y)\leq N$, then $\dist(x,y)\leq M$. A function is (metrically) \textit{proper} if the inverse image of bounded sets are bounded. A \textit{coarse} map is one that is bornologous and proper. Two functions $f:X\to Y$ and $g:X\to Y$ are \textit{close} if there is some constant $K>0$ with $\dist(f(x),g(x))\leq K$ for all $x\in X$.

The objects of the coarse category are metric spaces and the morphisms are close equivalence classes of coarse maps. Thus two metric spaces $X$ and $Y$ are coarsely equivalent if there are coarse maps $f:X\to Y$ and $g:Y\to X$ such that $g\circ f$ is close to the identity function on $X$ and $f\circ g$ is close to the identity function on $Y$.

Suppose $N>0$. Given a metric space $X$ and a basepoint $x_0\in X$, an ${N}$\textit{-sequence} in $X$ based at $x_0$ is an infinite list $(x_i)=x_0,x_1,\ldots$ of points in $X$ such that $\dist(x_i,x_{i+1})\leq N$ for all $i\geq 0$. We can interpret bornologous functions as functions such that for each $N>0$ there is an $M>0$ such that $N$-sequences in $X$ are sent to $M$-sequences in $Y$.

The invariant from \cite{borno}, \cite{coarse}, and \cite{direct} is defined in terms of $N$-sequences. Since we are studying the coarse properties of spaces we are only interested in sequences that go to infinity. An $N$-sequence $(x_i)$ goes to infinity if $\lim_{i\to \infty}\dist(x_0,x_i)=\infty$.

In this paper (Proposition \ref{CoarseSequences}) we note that the set of all sequences that are $N$-sequences in $X$ for some $N>0$ that go to infinity are precisely the functions $f:\mathbb N\to X$ that are coarse. Thus we call such sequences \textit{coarse sequences}. We are interested in an equivalence relation between coarse sequences that determines if they are ``going the same way'' to infinity. Requiring that the functions are close is too restrictive. For example, we would want $(n)$ and $(2n)$ in $\mathbb R$ to be equivalent.

In \cite{borno} the equivalence relation is taken to be the symmetric transitive closure of the relation ``subsequence.'' In that paper the sequences are restricted to $N$-sequences for a fixed $N>0$ based at a point $x_0$ to give the set of equivalence classes $\sigma_N(X,x_0)$. The invariant is then the direct limit $\varinjlim\sigma_N(X,x_0)$. The case where $\{\sigma_N(X,x_0)\}$ stabilizes is treated in \cite{borno} and \cite{coarse}. The full structure of the direct sequence is exploited in \cite{direct}.

Given this new view of these sequences as coarse sequences we give the definition independent of $N$. Given coarse sequences $s$ and $t$ in $X$ starting at $x_0$, relate $s$ to $t$ if $s$ is a subsequence of $t$. Denote the symmetric transitive closure of this relation by $\approx$.

\begin{definition}\label{SigmaDef}
Define $\sigma(X,x_0)$ to be the set of equivalence classes of coarse sequences in $X$ starting at $x_0$ under $\approx$.
\end{definition}

Since we are using the symmetric transitive closure, if $s\approx t$ then there are coarse sequence $s=s_1,s_2,\ldots,s_n=t$ such that for each $i$, either $s_i$ is a subsequence of $s_{i+1}$ or $s_{i+1}$ is a subsequence of $s_i$. Since the list is finite we can always find a $K>0$ so that each $s_i$ is a $K$-sequence. Thus this definition is equivalent to the one in \cite{direct}. In that paper it is shown that the set $\sigma(X,x_0)$ is independent of basepoint and invariant under coarse equivalences.

\section{Sequential ends}

Let $X$ be a topological space. A ray in $X$ is a continuous map $r:[0,\infty)\to X$. The following definition of ends is from \cite{BH}.

\begin{definition*}
Two (topologically) proper rays $r_1$ and $r_2$ are equivalent (converge to the same end) if for every compact subset $C$ of $X$ there is $N\in \mathbb N$ such that $r_1[N,\infty)$ and $r_2[N,\infty)$ are contained in the same path component of $X-C$.
\end{definition*}

We call the equivalence class of a proper ray $r$ an end of $X$ and denote it by $\en(s)$. Denote the set of ends in $X$ by $\Ends(X)$.

We form an analogous definition for the coarse category. First note the coarse analog to $[0,\infty)$ is $\mathbb N$ (we often take $\mathbb N=\{0,1,2,\ldots\}$ for convenience). Thus the analog to a ray should be a coarse function $f:\mathbb N\to X$.  But these are precisely $N$-sequences that go to infinity.

\begin{proposition}\label{CoarseSequences}
Suppose $s$ is a sequence in a metric space $X$. Write the sequence $s$ as a function $f:\mathbb N\rightarrow X$ by setting $f(n)=s_n$.

\begin{enumerate}
\item $s$ is an $N$-sequence for some $N>0$ if and only if $f$ is bornologous.
\item $s$ goes to infinity if and only if $f$ is proper.
\end{enumerate}
Thus $s$ is an $N$-sequence for some $N>0$ that goes to infinity if and only if $f$ is coarse.
\end{proposition}
\begin{proof}
\text{ }

\begin{enumerate}
\item $(\Rightarrow)$ We have an $N$-sequence $s$ for some $N>0$. We show for all $M>0$ there exists $R>0$ such that if $|i-j|\leq M$, then $d(f(i), f(j))\leq R$. Suppose $M>0$. Set $R=MN$. Without loss of generality assume $i<j$ and suppose $j-i\leq M$. Since $s$ is an $N$-sequence, $d(s_k,s_{k+1})\leq N$ for all $k\in\mathbb N$. Then $(f(i),f(j))=d(s_i,s_j) \leq MN$.

\noindent $(\Leftarrow)$ We have a bornologous function $f$ written $f:\mathbb N\rightarrow X$. There is $N$ such that if $|i-j|\leq 1$ then $d(f(i),f(j))\leq N$. Then $d(s_i,s_{i+1})\leq N$, thus $s$ is an $N$-sequence.

\item We prove the contrapositive.

\noindent $(\Rightarrow)$ Suppose $s$ does not go to infinity. Then there exists $R>0$ such that for all $M>0$ there exists $i_M\geq M$ so that $d(x_0,s_{i_M})\leq R$. Then $\{s_{i_M}\}_{M\in \mathbb N}$ is bounded, but the inverse image under $f$ contains $\{i_M\}_{M\in \mathbb N}$ and thus is unbounded.

\noindent $(\Leftarrow)$ Suppose $f$ is not proper. Then there is a bounded subset of $X$ whose inverse image is unbounded in $\mathbb N$. Thus $s$ has a bounded subsequence and does not go to infinity.
\end{enumerate}
\end{proof}

\begin{definition}
A sequence $(s_i)_{i\in\mathbb N}$ in $X$ is a coarse sequence if $f:\mathbb N\to X$ is coarse where $f(n)=s_n$.
\end{definition}

The analog to a path in $X$ is a $K$\textit{-chain}, a finite sequence $c_1,\ldots,c_n$ such that $\dist(c_i,c_{i+1})\leq K$. Two points $x,y\in X$ are in the same $K$-chain component if they can be joined by a $K$-chain. The set of $K$-chain components is a partition of $X$.

We now formulate our analogous definition of ends. We fix our attention to sequences that start at a fixed basepoint $x_0\in X$. The based theory is equivalent to unbased (see Proposition \ref{Unbased}).

\begin{definition}\label{SequentialEnds}
Two coarse sequences $s$ and $t$ starting at $x_0$ converge to the same end if there is a $K>0$ such that for all $R>0$ there is an $N\in\mathbb N$ such that $\{s_i\}_{i\geq N}$ and $\{t_i\}_{i\geq N}$ are contained in the same $K$-chain component of $X-\B(x_0,R)$.
\end{definition}

This defines an equivalence relation on the set of coarse sequences starting at $x_0$. Denote the equivalence class of such a sequence $s$ by $\en(s)$. We will show that the set of equivalence classes is identical to $\sigma(X,x_0)$ (Theorem \ref{MainTheorem}). Thus we use this notation for the set of sequential ends.

We state an equivalent formulation which resembles \cite[Lemma I.8.28]{BH}. The proof is left as an exercise.

\begin{proposition}\label{LevelChains}
Two coarse sequences $s$ and $t$ starting at $x_0$ converge to the same end if and only if there is a $K>0$ such that for all $R>0$ there is an $N\in\mathbb N$ so that for all $n\geq N$, $s_n$ and $t_n$ can be joined by a $K$-chain outside $\B(x_0,R)$.
\end{proposition}

In fact we do not need all $n\geq N$. 

\begin{lemma}
Suppose $s$ and $t$ are coarse sequences that start at $x_0$. Suppose there is $K>0$ such that for all $R>0$ there are $i,j\in\mathbb N$ such that $s_i$ and $t_j$ can be joined by a $K$-chain outside $\B(x_0,R)$. Then for all $R>0$ there is an $N\in\mathbb N$ such that for all $n\geq N$, $s_n$ and $t_n$ can be joined by a $K$-chain outside $\B(x_0,R)$.
\end{lemma}

\begin{proof}
Suppose $s$ is an $N_1$-sequence and $t$ is an $N_2$-sequence. Fix $K>0$ so that the hypothesis is satisfied. We can assume $K\geq\max\{N_1,N_2\}$. Suppose to the contrary that there is an $R>0$ such that for all $N\in\mathbb N$ there is an $n\geq N$ such that $s_n$ and $t_n$ cannot be joined by a $K$-chain outside $\B(x_0,R)$. Choose $i,j\in\mathbb N$ such that $s_i$ and $t_j$ can be joined by a $K$-chain outside $\B(x_0,R)$. Let $c_1,\ldots,c_k$ be such a chain. Now there is an $n\geq\max\{i,j\}$ so that $s_n$ and $t_n$ cannot be joined by a $K$-chain outside $B(x_0,R)$. Consider the $K$-chain $s_n,\ldots,s_{i+1},c_1,\ldots,c_k,t_{j+1},\ldots,t_n$.  Since $s_n$ and $t_n$ cannot be joined by a $K$-chain outside $\B(x_0,R)$, either $s_{i+1},\ldots,s_n$ or $t_{j+1},\ldots,t_n$ must travel inside $\B(x_0,R)$. 

Let $R_1=\max\{d(x_0,s_k)\}_{k\leq n}\cup\{d(x_0,t_k)\}_{k\leq n}$. Note $R_1\geq R$ since $i,j\leq n$ and $s_i$ and $t_j$ must lie outside $\B(x_0,R)$. Choose $s_{i_1}$ and $t_{j_1}$ that can be joined by a $K$-chain outside $\B(x_0,R_1)$. Note ${i_1},{j_1}>n$. Now there is an $n_1\geq\max\{i_1,j_1\}$ so that $s_{n_1}$ and $t_{n_1}$ cannot be joined by a $K$-chain outside $\B(x_0,R)$, thus either $s_{i_1+1},\ldots,s_{n_1}$ or $t_{j_1+1},\ldots,t_{n_1}$ must travel inside $\B(x_0,R)$. 

By induction we see that either $s$, $t$, or both cannot go to infinity.
\end{proof}

Thus we have another formulation with a weaker condition so it is convenient in the if direction.

\begin{proposition}\label{OnePoint}
Two coarse sequences $s$ and $t$ starting at $x_0$ converge to the same end if and only if there is a $K>0$ such that for all $R>0$ there are $i,j\in\mathbb N$ so that $s_i$ and $t_j$ can be joined by a $K$-chain outside $\B(x_0,R)$.
\end{proposition}

We now state an unbased version of the sequential ends that more closely matches the definition from \cite{BH}.

\begin{definition}
Two coarse sequences $s$ and $t$ in $X$ converge to the same end if there is a $K>0$ such that for all bounded sets $B\subset X$ there is an $N\in\mathbb N$ such that $\{s_i\}_{i\geq N}$ and $\{t_i\}_{i\geq N}$ are contained in the same $K$-chain component of $X-B$.
\end{definition}

We state equivalent formulations similar to Propositions \ref{LevelChains} and \ref{OnePoint} without proof.

\begin{proposition}\label{Unbased}
Suppose $s$ and $t$ are two coarse sequences in $X$. The following are equivalent.

\begin{enumerate}
\item $s$ and $t$ converge to the same end.

\item There is a $K>0$ such that for all bounded sets $B\subset X$ there is an $N\in\mathbb N$ so that for all $n\geq N$, $s_n$ and $t_n$ can be joined by a $K$-chain outside $B$.

\item There is a $K>0$ such that for all bounded sets $B\subset X$ there is an $n\in\mathbb N$ so that $s_n$ and $t_n$ can be joined by a $K$-chain outside $B$.

\item\label{UnbasedProp4} There is a $K>0$ such that for all bounded sets $B\subset X$ there are $i,j\in\mathbb N$ so that $s_i$ and $t_j$ can be joined by a $K$-chain outside $B$.
\end{enumerate}
\end{proposition}

Let $\sigma(X)$ be the set of (unbased) sequential ends in $X$. Since a based end $\en(s)\in \sigma(X,x_0)$ is naturally an element of $\sigma(X)$ we have a function $\sigma(X,x_0)\to \sigma(X)$. The proof of the following proposition is left to the reader.

\begin{proposition}
The natural function $\sigma(X,x_0)\to \sigma(X)$ is a well defined bijection.
\end{proposition}

The unbased version is more convenient when stating functorial properties.  

\begin{proposition}
Any coarse map $f:X\rightarrow Y$ induces a well defined function $f^*: \sigma(X)\rightarrow \sigma(Y)$ where $f^*(\en(s))=\en(f(s))$.
\end{proposition}

\begin{proof}
Suppose $\en(s)=\en(t)$ in $\sigma(X)$. We know there is $K>0$ such that for all bounded $B\subset X$ there is an $n\in\mathbb N$ so that $s_n$ and $t_n$ can be joined by a $K$-chain outside $B$. Since $f$ is bornologous, there is $M>0$ such that $K$-chains in $X$ are sent to $M$-chains in $Y$. Suppose $B\subset Y$ is bounded. To show $\en(f(s))=\en(f(t))$ in $\sigma(Y)$, we find $n\in\mathbb N$ such that $f(s_n)$ is joined to $f(t_n)$ by an $M$-chain outside $B$. We know $f$ is proper, so $f^{-1}(B)$ is bounded. We know there is an $n\in\mathbb N$ such that $s_n$ and $t_n$ can be joined by a $K$-chain outside $f^{-1}(B)$. But then $f(s_n)$ is joined to $f(t_n)$ by an $M$-chain outside $B$.
\end{proof}

Since the identity $X\to X$ induces the identity $\sigma(X)\to\sigma(X)$ and if $f:X\to Y$ and $g:Y\to Z$ are coarse functions then $(g\circ f)^*\equiv g^*\circ f^*$ we have a functor from the category of coarse metric spaces to the category of sets. 

We now have a nice proof that the cardinality of the sequential ends is invariant under coarse equivalence that is independent of the invariant in \cite{direct}.

\begin{theorem}
Suppose $X$ and $Y$ are coarsely equivalent. Then there is a bijection between $\sigma(X)$ and $\sigma(Y)$.
\end{theorem}

\begin{proof}
Suppose $f:X\to Y$ and $g:Y\to X$ comprise the coarse equivalence. We show that $f_*:\sigma(X)\to \sigma(Y)$ and $g_*:\sigma(Y)\to \sigma(X)$ compose to form the respective identities. Suppose $\en(s)\in\sigma(X)$. Since $g\circ f$ is close to the identity on $X$, $g_*\circ f_*(\en(s))=\en(f\circ g(s))=\en(s)$. The other composition is symmetric.
\end{proof}

For proper geodesic metric spaces the sequential ends and classical ends are equivalent.

\begin{proposition}\label{Ends}
Suppose $X$ is a proper geodesic space. Then $\Ends(X)$ and $\sigma(X)$ are equivalent as sets.
\end{proposition}

\begin{proof}
We wish to define a function $\Ends(X)\to\sigma(X)$. Suppose $r$ is a proper ray in $X$. Since $r$ restricted to $[0,1]$ is uniformly continuous there is a finite list $0=t_0<t_1<\cdots<t_n=1$ of real numbers such that for all $i<n$, $\dist(r(t_i),r(t_{i+1}))\leq 1$. We define the first $n+1$ terms of our desired sequence to be $r(t_i)$, $0\leq i\leq n$. Then we repeat the procedure for $r$ restricted to $[1,2]$ to find the a list of real numbers $1=s_0<s_1<\cdots <s_m=2$ with $\dist(r(s_i),r(s_{i+1}))\leq 1$ and define the next $m$ terms of our sequence to be $r(s_i)$, $0<i\leq m$. By induction we have a $1$-sequence in $X$. It is a coarse sequence since a bounded set in $X$ is contained in a compact set so its inverse image in $[0,\infty)$ is compact and therefore bounded and thus corresponds to a finite number of terms of our sequence. In view of \cite[Lemma I.8.28]{BH} this function $\Ends(X)\to\sigma(X)$ is well defined. 

We now define a function $\sigma(X)\to\Ends(X)$. Suppose $s$ is a coarse sequence. For each $i\in\mathbb N$ choose a geodesic from $s_i$ to $s_{i+1}$. Concatenate these geodesics to form a ray $r:[0,\infty)\to X$ such that $r(n)=s_{n}$ for all $n\in\mathbb N$. It is proper since a compact set of $X$ is contained in a bounded set so the inverse image of $s$ contains finitely many terms and thus corresponds to a compact subset of $[0,\infty)$. Again in view of \cite[Lemma I.8.28]{BH} this function is well defined.

Consider the composition $\Ends(X)\to\sigma(X)\to\Ends(X)$. Given a proper ray $r$ in $X$, the ray assigned to $r$ is a ray $r'$ such that for each $n\in\mathbb N$ there is $m\geq n$ with $r'(m)=r(n)$. Given a compact set $C\subset X$ there is $N\in\mathbb N$ such that $r[N,\infty)$ and $r'[N,\infty)$ lie outside $C$. Then we may choose $m>n\geq N$ with $r'(m)=r(n)$ so $r[N,\infty)$ and $r'[N,\infty)$ lie in the same path component of $X-C$.

Now consider the composition $\sigma(X)\to\Ends(X)\to\sigma(X)$. Given a coarse sequence $s$ in $X$, the sequence assigned to $s$ is a sequence $s'$ such that for each $n\in\mathbb N$ there is $m\geq n$ such that $s'_m=s_n$.  given a bounded set $B\subset X$ there is $N\in\mathbb N$ so that $\{s_i\}_{i\geq N}$ and $\{s'_i\}_{i\geq N}$ lie outside $B$. Then we may choose $m>n\geq N$ with $s'_m=s_n$ so $\en(s')=\en(s)$ by Proposition \ref{Unbased}(\ref{UnbasedProp4}). Note we can see that $s'$ is a subsequence of $s$ so they converge to the same end by Theorem \ref{MainTheorem}.

\end{proof}

\section{Equivalence between the sequential ends and $\sigma(X,x_0)$}

Before proving that the invariant from \cite{direct} is equivalent to the sequential ends of a space we offer a simplified definition of the invariant from \cite{direct}. We show that if $s\approx t$ then there is a single sequence that contains both $s$ and $t$ as subsequences.

\begin{lemma}\label{ReplacingSequences}
Suppose $s$, $t$, and $r$ are $N$-sequences in $X$ going to infinity based at $x_0$. Suppose that $t$ is a subsequence of $s$ and $r$ is a supersequence of $t$. Then there is an $N$-sequence $r'$ based at $x_0$ going to infinity that is a supersequence of both $r$ and $s$.
\end{lemma}

\begin{proof}
We start at the basepoint $x_0$. Suppose the terms $s_1,\ldots,s_n$ of $s$ are removed between $t_0=x_0$ and $t_1$ when passing from $s$ to $t$. Also suppose that the terms $r_1,\ldots,r_m$ are added between the terms $t_0$ and $t_1$ when passing from $t$ to $r$. We define the first several terms of $r'$ to be $t_0,s_1,\ldots,s_n,t_1,t_0,r_1,\ldots,r_m,t_1$. Note if no terms $s_1,\ldots,s_n$ are removed then $r'$ will go right from $t_0$ to $t_1$. Likewise if no terms $r_1,\ldots,r_m$ are added then there is no need to travel back to $t_0$; we simply end at $t_1$.

We repeat the same procedure starting at $t_1$. By induction we can define the $N$-sequence $r'$ that is a supersequence of both $s$ and $r$. Notice that $r'$ goes to infinity since $s$, $t$, and $r$ do.
\end{proof}

\begin{proposition}\label{OneSupersequence}
If $s$ and $t$ are equivalent $N$-sequences in $X$ based at $x_0$ that go to infinity then there is an $N$-sequence $r$ in $X$ based at $x_0$ that goes to infinity that is a supersequence of $s$ and $t$.
\end{proposition}

\begin{proof}
Let $s=s_1,s_2,\ldots,s_n=t$ be $N$-sequences $s_i$ in $X$ based at $x_0$ that go to infinity that realize the equivalence between $s$ and $t$. We may assume that any time $s_i$ is a subsequence of $s_{i-1}$ that $s_{i+1}$ is a supersequence of $s_i$ and vice versa. Suppose $s_2$ is a subsequence of $s_1$. Then we replace $s_3$ by a sequence $s_3'$ that is a supersequence of both $s_1$ and $s_3$ according to Lemma \ref{ReplacingSequences}. Repeating this process we arrive at our conclusion--if $n$ is odd then $s_n'$ is a supersequence of both $s_n$ and $s_1$ and if $n$ is even then $s_{n-1}'$ is a supersequence of both $s_n$ and $s_1$.

Now suppose $s_2$ is a supersequence of $s_1$. We replace $s_4$ by a sequence $s_4'$ that is a supersequence of $s_2$ (and therefore $s_1$) and $s_4$. Repeating this process we again arrive at our conclusion. If $n$ is even then $s_n'$ is our desired sequence and if $n$ is odd $s_{n-1}'$ is our desired sequence.
\end{proof}

We can now prove the main theorem.

\begin{theorem}\label{MainTheorem}
Let $s$ and $t$ be coarse sequences. Then $s\approx t$ if and only if $\en(s)=\en(t)$.
\end{theorem}

\begin{proof}
$(\Rightarrow)$ We have a $K$-sequence $r$ containing $s$ and $t$ as subsequences. Suppose $R>0$. Then there exists $N\in\mathbb N$ so that $d(r_n,x_0)\geq R$ for all $n\geq N$. There is $i\geq N$ so that $r_i=s_{k_1}$ and there is $j\geq N$ with $r_j=t_{k_2}$ for some $k_1,k_2\in\mathbb N$. Then $s_{k_1}$ and $t_{k_2}$ are joined by the $K$-chain that runs along $r$ from $r_i$ to $r_j$.

$(\Leftarrow)$ For all $i \in \mathbb N$ there is $N_i>0$ so that for all $n\geq N_i$ there is a $K$-chain $c_i$ from $s_n$ to $t_n$ outside $\B(x_0,i)$. We can assume $N_{i+1}>N_i$. We create a $K$-sequence $r$ containing $q$ and $t$ as subsequences inductively. We start at $x_0$ and follow $t$ to $t_{N_1}$ and back to $x_0$, then follow $s$ to $s_{N_1}$. Then at the $i$th stage we start at $s_{N_i}$ and follow $c_i$ to $t_{N_i}$. Then we travel along $t$ to $t_{N_{i+1}}$ and back to $t_{N_i}$ then back along $c_i$ to $s_{N_i}$ and then along $s$ to $s_{N_{i+1}}$. Since $s$ and $t$ go to infinity and $c_i$ lies outside $\B(x_0,i)$ we see that $r$ must go to infinity.
\end{proof}

\section{Examples}\label{Examples}

\begin{example}\label{sigmaR}
The real line $\mathbb R$ was explored in \cite[Theorem 3.6]{borno}. We revisit the example. We show that $\sigma(\mathbb R,0)=\{\en((n)),\en((-n))\}$. Certainly $\en((n)),\en((-n))\in \sigma(\mathbb R,0)$. First we show that $\en((n))\neq \en((-n))$. Suppose $K>0$. Set $R=K$. Then for all $i,j\in\mathbb N$, $i$ and $-j$ cannot be joined by a $K$-chain outside $\B(0,R)$.

Now we show that $\sigma(\mathbb R,0)\subset \{\en((n)),\en((-n))\}$. Suppose $s$ is an coarse sequence in $\mathbb R$ starting at $0$. There is an $M\in\mathbb N$ so that $\dist(0,s_i)>N$ for all $i\geq M$.

There are two cases: $s_M>0$ and $s_M<0$. We treat the first case; the second case is symmetric to the the first. We will show that $\en(s)=\en((n))$. Since $s_M>0$ and $d(s_i,0)>N$ for all $i\geq M$, $s_i>0$ for all $i\geq M$. Given $R>0$ there is $T>0$ so that for all $i\geq T$, $d(0,s_i)>R$. We can assume $T\geq M$. Let $n$ be an integer greater than $R$. Then $s_T$ and $n$ can be joined by a 1-chain outside $\B(0,R)$. Thus $\en(s)=\en((n))$.

\end{example}

\begin{example}
We show that $|\sigma (\mathbb R^n, \bm 0)|=1$ for $n\geq 2$. We could use Proposition \ref{Ends} and the fact that $\mathbb R^n-\B(\bm 0,R)$ is path connected. Instead we explicitly show any coarse sequence $s$ in $\mathbb R^n$ that starts at $\bm 0$ converges to the same end as the sequence $((m,0,\ldots,0))$. The proof illustrates the use of the extra dimension to form $K$-chains.

We use the max metric. Given $R>0$ there is an $k\in\mathbb N$ such that $\dist(s_k,\bm 0)>R$. Choose an integer $m>R$. We construct a 1-chain in $\mathbb R^n$ from $s_k$ to $(m,0,\ldots 0)$. Write $s_k=(x_1,\ldots,x_n)$. Then there is a $j\leq n$ with $|x_j|>R$. There are two cases: $j>1$ and $j=1$.

\textit{Case 1:} $j>1$. Choose a 1-chain in $\mathbb R$ from $x_1$ to $m$ and define a 1-chain in $\mathbb R^n$ to start at $s_k$ and follow that chain in the first coordinate to the point $(m,x_2,\ldots,x_n)$. Since $j>1$ and $|x_j|>R$ each term of this chain stays outside of $\B(\bm 0,R)$. 

For each $i>0$ choose a 1-chain in $\mathbb R$ from $x_i$ to $0$ and follow it in the $i$th coordinate from the point $(m,0,\ldots,x_i,x_{i+1},\ldots,x_n)$ to the point $(m,0,\ldots,0,x_{i+1},\ldots,x_n)$. This 1-chain stays outside of $\B(\bm 0,R)$ since $m>R$.

We concatenate these chains to form a 1-chain from $s_k$ to $(m,0,\ldots 0)$.

\textit{Case 2:} $i=1$. We must be careful to choose a 1-chain from $s_k$ to $(m,x_2,\ldots,x_n)$ that stays outside $\B(\bm 0, R)$. We do so by choosing a 1-chain in $\mathbb R$ from $x_2$ to $m$ and following it from $s_k$ to $(x_1,m,x_3,\ldots,x_n)$. Then we follow a 1-chain in $\mathbb R$ from $x_1$ to $m$ to go from $(x_1,m,x_3,\ldots,x_n)$ to $(m,m,x_3,\ldots,x_n)$ and then finally back to $(m,x_2,\ldots,x_n)$. The rest of the construction proceeds as in Case 1.
 
\end{example}

The previous examples were geodesic so the sequential ends correspond to the classical ends. We now give several examples of spaces that are not geodesic.

\begin{example}
Let $X=\{(-1,y):y\geq 1\}\cup \{(x,1):-1\leq x\leq 1\}\cup \{(1,y):y\geq 1\}\subset \mathbb R^2$, an example from \cite{borno} and \cite{coarse}. We can easily see that $\sigma(X)=1$ while $\Ends(X)=2$.
\end{example}

\begin{example}
Let $X$ be the space consisting of tangent circles of radius $r=2^n$ with center $(0,3\cdot2^n)$ for $n=0,1,2,\ldots$. We show $|\sigma(X,(0,2))|=1$. Let $s$ be a $1$-sequence in $X$ that starts at $(0,2)$ and travels along the right of every tangent circle and then up to the next point of tangency. We show that any coarse sequence $t$ is equivalent to $s$. Suppose $t$ is a coarse sequence in $X$. Given $R>0$ we can find $t_i$ such that $t_i$ lies on a circle that lies outside $\B((0,2),R)$. Since $t_i$ lies on one of the circles and we know there is a point on $s$, say $s_j$ on that circle, $t_i$ and $s_j$ can be joined by a $1$-chain on that circle. Thus $\en(t)=\en(s)$. 
\end{example}

\begin{example}
Let $X=\{(x,2^n): x\in \mathbb R, n\in \mathbb N\} \cup \{(0,y): y\geq 1\}$. We show $|\sigma(X, (0,1))|=3$. Suppose $t$ is a coarse sequence in $X$ starting at $(0,0)$ and set $t_i=(x_i,y_i)$. There are two cases, $(y_i)$ is either bounded or unbounded. 

First suppose $(y_i)$ is unbounded. We show $\en(t)=\en((0,n))$. Suppose $R>0$. Choose $i$ such that $y_i>R$ and choose $j>R$. Then $t_i$ and $(0,j)$ can be joined by a 1-chain outside $\B(x_0,R)$. 
 
Now suppose $(y_i)$ is bounded by $M$. Now $(y_i)$ is an $N$-sequence for some $N>0$; we can assume $M>N$. Choose $S\in\mathbb N$ so $d(t_i,(0,0))>M$ for all $i\geq S$. We have two more cases, $x_S>0$ and $x_S<0$. We treat $x_S>0$, the other is symmetric. We show $\en(t)=\en((n,0))$.

We use the max metric. Since $y_i\leq M$ and $d((x_i,y_i),(0,0))>M$ for all $i\geq S$, we have $d((x_i,y_i),(0,0))=d(x_i,0)$ for all $i\geq S$. Also, since $d((x_i,y_i),(x_{i+1},y_{i+1}))\leq N$ for all $i\geq 0$, $d(x_i,x_{1+1})\leq N$ for all $i\geq 0$. Then, since $x_S>0$ and $d(x_i,0)>M>N$ for all $i\geq S$, $x_i>0$ for all $i\geq S$.  Given $R>0$, there is $T>0$ so that for all $i\geq T$, $d(0,x_i)>R$. We can assume $T\geq S$. Let $n$ be an integer greater than $R$. Then $q_T$ and $(n,0)$ can be joined by an $M$-chain outside $\B((0,0),R)$.
\end{example}

\end{document}